\documentclass[letterpaper, 10pt, journal]{IEEEtran}

\usepackage{array}
\usepackage{textcomp}
\usepackage{stfloats}
\usepackage{url}
\usepackage{verbatim}
\usepackage{graphicx}
\usepackage{cite}
\usepackage{url}

\usepackage[utf8]{inputenc}

\usepackage{graphicx} %package to manage images
\graphicspath{ {./images/} }
\usepackage{color}
\usepackage{amsmath,amsthm,amssymb,amsfonts}
\usepackage{comment}
\usepackage{mathtools,bbm}
\usepackage[rightcaption]{sidecap}
\usepackage{wrapfig}
\usepackage{subcaption}
\usepackage[font=footnotesize,labelfont=footnotesize]{caption}
\usepackage{url}
\usepackage{nomencl}
\makenomenclature
\usepackage{etoolbox}
\renewcommand\nomgroup[1]{%
  \item[\bfseries
  \ifstrequal{#1}{A}{Sets}{%
  \ifstrequal{#1}{B}{Decision variables}{%
  \ifstrequal{#1}{C}{Parameters}{}}}%
]}
\usepackage[ruled,shortend,commentsnumbered,linesnumbered]{algorithm2e}
\usepackage{multirow}

\newcommand{\setr}[1]{\{#1\}}
\newcommand{\abs}[1]{\ensuremath{\left\lvert{#1}\right\rvert}}
\newcommand{\st}{\operatorname{subject \text{$\, \,$} to}}
\renewcommand{\st}{\operatorname{s.} \operatorname{t.}}
\newcommand{\wrt}{\operatorname{w.} \operatorname{r.} \operatorname{t.}}
\newcommand{\Pt}{\widetilde{P}}
\newcommand{\GG}{\mathcal{G}}
\newcommand{\WW}{\mathcal{W}}
\newcommand{\DD}{\mathcal{D}}
\newcommand{\LL}{\mathcal{L}}
\newcommand{\CC}{\mathcal{C}}
\newcommand{\MM}{\mathcal{M}}
\renewcommand{\SS}{\mathcal{S}}
\newcommand{\Qb}{\mathbb{Q}}
\newcommand{\Qbhat}{\widehat{\mathbb{Q}}}

\newcommand{\norm}[1]{\|#1\|}
\newcommand{\dist}{\mathrm{dist}}

\newcommand{\bt}{\tilde{b}}
\newcommand{\bo}{\overline{b}}
\newcommand{\bu}{\underline{b}}
\newcommand{\ndr}{n^{\mathrm{DR}}}

\newcommand{\setdef}[2]{\{#1 \; | \; #2\}}
\newcommand{\setdefb}[2]{\Big\{#1 \; | \; #2 \Big\}}
\newcommand{\real}{\mathbb{R}}
\newcommand{\realnonnegative}{\mathbb{R}_{\ge 0}}
\newcommand{\Pb}{\mathbb{P}}
\newcommand{\Eb}{\mathbb{E}}

\newcommand{\mis}{\mathrm{mis}}
\newcommand{\out}{\mathrm{con}}
\newcommand{\costcoeff}{\mathsf{H}}
\newcommand{\Xirob}{\Xi_{\mathrm{rob}}}
\newcommand{\Xihat}{\widehat{\Xi}}
\newcommand{\until}[1]{[#1]}
\newcommand{\uncertain}{\xi}
\newcommand{\uncertainhat}{\widehat{\uncertain}}

\usepackage[normalem]{ulem} % strike out text, example:  \sout{Hello}

\definecolor{new}{rgb}{0.55,0,0.55}
\usepackage[prependcaption,colorinlistoftodos]{todonotes}

\newtheorem{theorem}{Theorem}[section]

\newtheorem{remark}[theorem]{Remark}

\newtheorem{proposition}[theorem]{Proposition}  

\usepackage{xcolor}

\newcommand{\oprocendsymbol}{\hbox{$\bullet$}}
\newcommand{\oprocend}{\relax\ifmmode\else\unskip\hfill\fi\oprocendsymbol}
\newcommand{\longthmtitle}[1]{\mbox{}\textup{\textsl{(#1):}}}

\allowdisplaybreaks

\begin{document}
\title{A two-step approach to Wasserstein distributionally robust chance- and security-constrained dispatch}

\author{Amin Maghami, Evrim Ursavas, Ashish Cherukuri
\thanks{A. Maghami and A. Cherukuri are with the Engineering and Technology institute Groningen, Unversity of Groningen. E. Ursavas is with the Faculty of Economics and Business, University of Groningen. (e-mails: \texttt{ \{a.maghami, a.k.cherukuri, e.ursavas\}@rug.nl}). Amin Maghami is  supported by a scholarship from the Faculty of Science and Engineering, University of Groningen. }
}

\maketitle
 
\begin{abstract}
    This paper considers a security constrained dispatch problem involving generation and line contingencies in the presence of the renewable generation. The uncertainty due to renewables is modeled using joint chance-constraint and the mismatch caused by contingencies and renewables are handled using reserves. We consider a distributionally robust approach to solve the chance-constrained program. We assume that samples of the uncertainty are available. Using them, we construct a set of distributions, termed ambiguity set, containing all distributions that are close to the empirical distribution under the Wasserstein metric. The chance constraint is imposed for all distributions in the ambiguity set to form the distributionally robust optimization problem. This problem is nonconvex and computationally heavy to solve exactly. We adopt a two-step approach to find an approximate solution.  In the first step, we construct a polyhedral set in the space of uncertainty that contains enough mass under all distributions in the ambiguity set. This set is constructed by solving several two-dimensional distributionally robust problems.  In the second step, we solve a linear robust optimization problem where the uncertain constraint is imposed for all uncertainty values lying in the polyhedral set. We demonstrate the scalability and robustness of our method using numerical experiments.
\end{abstract}
%\mbox{}
%

\section{Introduction}
Power systems are going through important changes, driven mainly by the increasing penetration of renewable sources. While using renewable energy is vital for mitigating climate change, they also bring significant challenges to the power grid. Renewable sources such as wind and solar are weather dependent resulting in uncertainty in production and difficulty in control. Furthermore, the intermittent nature of the production makes the grid more vulnerable to component failures rising the need for increased attention to secure the grid under such contingencies. Consequently, grid operators must clearly understand and accurately model the uncertainty and variability of renewable generation in order to achieve stable and optimal operation of the future power systems.
In addition, a secure power system must be able to withstand all possible contingencies without violating physical limitations which may lead to cascading failures \cite{PK-JP-VA-GA-AB-CC-NH-DH-AS-CT-TV-VV:04}. To avoid this, corrective N-1 security criterion is widely used within the dispatch or Optimal Power Flow (OPF) models, according to which the system is capable of finding a new operational solution after the occurrence of a single contingency (for example a line or a generator contingency), while no operational security limits (such as line flow limits) are violated. Despite the importance of N-1 criterion, the security constrained OPF problem has not been widely studied in the presence of renewable sources, except for few works \cite{MV-KM-JL-GA:13,KS-HN-LR-SM-RB-DB:19,HS-NA-HZ:17,MV-SC-EI-MI-TK-OM-JM-LA-RW-GA:13,LR-FO-BV-GA:15,LR-FO-TK-GA:13}. These works consider renewables to be uncertain, model the dispatch as stochastic optimization problems, and look for solutions where either information about the distribution is (partially) known or is inferred through samples. However, they do not consider the inherent robustness issues that stem from ``uninformativity'' of data about the uncertainty. That is, when the data is either scarce, corrupted, or sampled in a non-identical manner. This challenge is overcome by employing techniques of distributionally robust (DR) optimization, where decisions are made taking into account worst-case perturbations to the collected data. Motivated by this fact, in this paper, we propose a Wasserstein metric based DR approach for solving a dispatch problem that takes into account N-1 security constraints and the influence of uncertain renewable  generation through chance-constraints. We then provide a two-step scalable method for approximating the solution of this DR problem.

\subsection{Related works}

A comprehensive survey~\cite{MV-SC-EI-MI-TK-OM-JM-LA-RW-GA:13} on the stochastic security constrained dispatch problem has showed that considering forecast error is important and employing corrective control actions can reduce the excessive cost and increase the integration of renewable sources. Often, such security-constrained stochastic dispatch problems involve uncertainty in constraints which can be modeled in several different ways, from chance constraints~\cite{KS-HN-LR-SM-RB-DB:19,YD-TM-MM:22} to a robust formulation~\cite{YY-PL-EL-TZ-JZ-FZ-DS:17, LY-CZ-JJ:18}. We focus on the most popular choice among these, that of chance-constrained problems, where the uncertain constraint is required to hold with high probability. One way of solving these problems in a data-driven way is the scenario approach~\cite{GC-MC:07} where the underlying uncertain constraint defining the chance-constraint is required to hold for all samples of the uncertainty. Such a method was studied in the context of security constrained dispatch in~\cite{FC-SF-PP-lW:12,FB-FG:08,MV-KM-JL-GA:13}.  

The scenario approach gets computationally burdensome when the dataset is large. To let go of this roadblock, the work~\cite{KM-PG-JL-14} explores a two-step procedure for solving chance-constraints, where sample-based constraint is replaced with a robust optimization. This is similar to what we explore in the distributionally robust setting in the current work. An alternative way of handling chance constraints is to use analytical reformulations when the distribution is Gaussian, e.g., in the case of DC OPF~\cite{LR-FO-TK-GA:13,KS-HN-LR-SM-RB-DB:19} and AC OPF~\cite{LR-GA:18,AV-LH-AB-LR-SC:20}. This is extended to mixture of Gaussians using linearization-based technique in~\cite{GC-HZ-YS:22}. However, no theoretical claims were made. 

All the above mentioned methods assume that either the probability distribution is known or samples drawn from it in an independent manner are available. The former and the latter assumptions are very strong, as pointed out in~\cite{ZW-QB-HX-DG:16} and~\cite{BP-AH-SB-DC-AC:20}, respectively. This fact motivates us to look at distributionally robust (DR) optimization, where decisions  hedge against the lack of complete information about the uncertainty by considering a set of distributions. The DR problems take different forms, depending on the type of the set of distributions, termed ambiguity set, considered in the optimization problem. In~\cite{LR-FO-BV-GA:15,YZ-SS-JM:17} moment-based ambiguity sets, that contain all distributions with a given first and second moment, are used. For such a setup, one still needs some data about the uncertainty to estimate the first and second moment that defines the ambiguity set. On the other hand, ambiguity sets constructed directly using data and a suitable distance metric in the space of distributions circumvent this modeling limitation. For these sets, the radius of the ambiguity set can be adjusted to control the level of conservativeness. In~\cite{YC-QG-HS-ZL-WW-ZL:18}, Kullback-Leibler (KL) divergence distance is used for DR unit commitment problem. For uncertainties with finite support, the $\ell_1-$ and $\ell_\infty-$norm based distance between distributions is employed in~\cite{CZ-YG:16,ZS-HL-VD:19,TD-QY-XL-CH-YY-MW-FB:19} for unit commitment problem and in~\cite{LX-CW-XS-ZB-LX:19} for security-constrained dispatch. Both KL and norm based metrics require the uncertainty to have finite support or only look for distributions supported on obtained historical data. The Wasserstein metric on the other hand can generalize to distributions supported on points that are not present in the historical data, thereby leading to better robustness performance. Recent works~\cite{XZ-HC:20,AZ-MY-MW-YZ:20,RZ-HW-XB:19,BP-AH-SB-DC-AC:20,AA-CO-JK-ZD-JT-FV:22} demonstrate the effectiveness of this metric in various dispatch related problems.Nevertheless, with Wasserstein metric, as the number of available samples increases or the network becomes large, solving the DR chance-constrained problem exactly becomes computationally difficult. This is because the DR problem is equivalent to a bilinear~\cite{AA-CO-JK-ZD-JT-FV:22} or a mixed-integer program~\cite{ZC-DK-WW-18,WX:21} with the number of constraints that scales with the product of the number of samples and size of the network. 
The main idea of our work is to explore methods that solve optimization problems having decision variables and constraints either independent of the size of the network or the number of samples. To this end, we build on the two-step approaches proposed in~\cite{CD-WF-LJ-LY-JL:18} and~\cite{BP-AH-SB-DC-AC:20}. The idea is to first construct a set in the space of uncertainty such that the probability mass contained in this set for all distributions in the ambiguity set is high. Then, in the second step, a robust optimization problem where the uncertain constraint is enforced for all uncertain values in the constructed set is solved. Our work differs from~\cite{CD-WF-LJ-LY-JL:18} and~\cite{BP-AH-SB-DC-AC:20} as we consider joint chance-constraints as opposed to individual ones in~\cite{CD-WF-LJ-LY-JL:18} and we take into account correlation among renewables when performing the first step, which is different from~\cite{BP-AH-SB-DC-AC:20}. In the latter comparison, our method is less conservative as we are able to construct more ``informative'' sets in the first step  while still retaining the tractability of the second step.

\subsection{Summary of contributions} 
\emph{Our first contribution} is modeling the dispatch problem considering renewables and contingencies. We consider a security constrained dispatch problem with two stages of decision making, where the stages are separated by an event consisting of the realization of power provision by the renewable sources and contingency or failure in the network. For the first stage,  the operator needs to determine for each conventional generator, the generation set point and the procurement of the reserve capacities.  The latter include two different sets of reserves, ones that are activated to handle imbalances due to uncertainty and the others to counter the contingency. The second stage decisions are control policies that determine how the power imbalance is adjusted network-wide using reserves. Concurrently, the aim is to ensure that the line flow and reserve capacity limits are satisfied collectively with high probability. These considerations give rise to a chance-constrained optimization problem. We assume that the information about the uncertainty is known through a certain number of samples. We consider the ambiguity set to be all distributions with a compact support that are close to the empirical distribution in the Wasserstein metric and formulate the DR chance-constrained problem where the probabilistic constraint is required to hold for all distributions in the ambiguity set. \emph{Our second contribution} is the two-step procedure to approximate the solution of the DR problem. In the first step, we assume that the renewables are partitioned into sets where two renewables are correlated only if they belong to a partition. This encapsulates the scenario where renewables in a certain geographical location will have similar pattern of energy generation. For each partition, we construct a polyhedral set that contains enough probability mass under all distributions in the ambiguity set. We construct these sets by solving several two-dimensional DR chance-constrained problems. In the second step, we do robust optimization with respect to the Cartesian product of all these polyhedral sets. The optimal solution of the robust problem is guaranteed to be feasible for the DR chance-constrained problem. Finally, \emph{our third contribution} involves demonstration of our two-step method with a numerical example of IEEE~RTS~24-bus system. We show how our method is scalable in terms of computational ease as the number of samples grow, while providing a tunable parameter to trade-off cost and violation frequency of the obtained decision. 

The rest of the paper is organized as follows. Section~\ref{sec:problem} models the optimization problem capturing the dispatch decision. Section~\ref{sec:method} considers the DR version of the problem, presents the two-step approach, derives computational tractability of the method, and analyzing the guarantee of the obtained solution. Finally, Section~\ref{sec:sims} presents the numerical example. 

\section{Problem statement}\label{sec:problem}

In this section, we define the N-1 security constrained dispatch problem in the presence of uncertainty caused due to renewable sources. The aim is to optimize the cost of  energy generation and reserve procurement while satisfying constraints that encode various capacity limits, reserve activation for handling discrepancies caused by contingencies and renewables, and probabilistic flow constraints due to line capacity limits.
Let \(\GG\), \(\WW\), \(\DD\), and \(\LL\) be the set of conventional generators, renewable sources, loads, and power lines, respectively. Each load $d \in \DD$ is associated with a demand $P_d \ge 0$ that is assumed to be known and fixed. Let $\CC := \setr{0,1,2,\dots,N_{\CC}}$ be the set enumerating the possible line and generator contingencies. Each element $c \in \CC$ with $c \ge 1$ represents a scenario where one component, either a line or a conventional generator has failed. The element $0 \in \CC$ corresponds to the case where the system is intact. We assume that the network remains connected under any line contingency. For notational ease, we divide the set of contingencies $\CC$ into ones where a generator fails $\CC_g \subset \CC$ and ones where a line fails $\CC_\ell \subset \CC$. Note that $\CC = \CC_g \cup \CC_\ell$ and $\CC_g \cap \CC_\ell = \emptyset$. For $c \in \CC_g$, the generator under contingency is denoted by $g_c \in \GG$. We consider two sets of decision variables, ones that are made before and after encountering contingency and uncertainty, respectively. The first set includes the planned power generation for the intact system, denoted $p_g$, $g \in \GG$, and two types of reserves that handle generation mismatch due to uncertainty and generator contingency, respectively. For the renewable uncertainty, we denote $r^+_g, r^-_g \ge 0$,  $g \in \GG$ as the up and down reserve capacities. The values $r^+_g$ and $r^-_g$ denote the upper and lower bound on the change in generation from the planned value $p_g$ that generator $g$ can execute using reserves. Moreover, $r_g^{\out} \ge 0$, $g \in \GG$ is the reserve capacity of $g$ to handle the change in power due to generator contingencies. We decide on the planned generation $p_g$ in such a way that it allows a feasible power flow under any line contingency. Thus, these events do not require reserve capacity.
We assume that we have access to the forecast power production by each renewable generator $P_w$ , $w \in \WW$ at the first stage, that is, before the actual uncertainty and contingency is revealed. The actual generation $\Pt_w$ usually deviates from the forecast value. To this end, we assume that the actual generation $\Pt = (\Pt_w)_{w \in \WW}$ is a random variable with distribution $\Pb$ and a compact support $\Xi \subset \realnonnegative^{n_w}$, where $n_w := \abs{\WW}$\footnote{The number of elements in a set $\SS$ is denoted by $\abs{\SS}$.}. We assume that the actual renewable generation as well as the contingency is realized before making the second set of decisions which involves deploying reserves to ensure power balance in a probabilistic sense. The total deviation of the actual and the forecast renewable generation in the network is represented by $\Pt_{\mis} = \sum_{w \in \WW}({P_w} - \Pt_w)$.
For the second-stage, we consider two sets of variables. To balance the mismatch due to renewable uncertainty, we let $\alpha^c_g$, $c \in \CC$, $g \in \GG$ be the affine control policies that determine the level of reserve activation. Given the total power mismatch $\Pt_\mis$, the change in power generation of $g$ using reserves under contingency $c \in \CC$ is $\alpha^c_g \Pt_\mis$. Thus, these policies define how $\Pt_\mis$ is handled collectively by generators active under contingency $c$. The second set of variables $\delta^c_g$, $c \in \CC$, $g \in \GG$ tackle the mismatch due to contingency.  Specifically, if $c \in \CC_\ell$, then there is no net power loss, so we set $\delta^c_g = 0$ for all $g \in \GG$. If $c \in \CC_g$, then the planned power $p_{g_c}$ is lost as generator $g_c$ fails, and for this case, the change in power generation is denoted by $\delta^c_g$. Note that for $g_c$, we set $\delta^c_{g_c} = - p_{g_c}$. Thus, the net power generated under contingency $c \in \CC$ by an active generator $g$ is $p_g + \Pt_\mis \alpha^c_g + \delta^c_g$. This is a random quantity as the mismatch is uncertain. In our formulation, the power balance holds for this eventually planned generation level  if it holds for $p_g$ under forecast renewable generation. However, the line constraints might only hold in a probabilistic manner. We next define our main optimization problem and term it as the chance-constrained N-1 security constrained dispatch (CCSCD)
\begin{subequations}\label{eq:SCOPF1}
\begin{align}
\min & \, \,  \sum_{g \in \GG} \Bigl( {{\costcoeff_g}{p_g} + {\costcoeff_g^+}{r^{+}_g}+{\costcoeff_g^-}{r^{-}_g}+{\costcoeff_g^{\out}}{r^{\out}_g}} \Bigr) \label{eq:SCOPF1-obj}
\\
 \wrt & \, \,  \{p_g,r_g^+,r_g^-,r_g^{\out}\}_{g \in \GG}, \{\alpha_g^c,\delta_g^c\}_{g \in \GG, c \in \CC}
\\
\st &  \, \,  \sum_{g \in \GG}{p_g} + \sum_{w \in \WW}{P_w} = \sum_{d \in \DD}{P_d},  \label{eq:SCOPF1-gen-eq}
\\
& \, \,  p_g + r_g^{+} +r_g^{\out} \leq P_g^{\max},  \quad  \forall g \in \GG,\label{eq:SCOPF1-gen-bound-up}
\\
& \, \,   p_g - r_g^{-} \geq P_g^{\min},  \quad \forall g \in \GG,\label{eq:SCOPF1-gen-bound-down}
\\
& \, \,  0 \leq r_g^{-}, r_g^{+},r_g^{\out} \leq R_g^{\max},  \quad \forall g \in \GG, \label{eq:SCOPF1-res-bound}
\\
& \, \,  \sum_{g \in \GG} \alpha^c_g=1, \quad \forall c \in \CC, \label{eq:SCOPF1-alpha-1}
\\
& \, \,  \alpha^c_g \ge 0,  \quad  \forall g \in \GG, \forall c \in \CC, \label{eq:SCOPF1-alpha-nn}
\\
& \, \,  \alpha^c_{g_c} = 0, \quad \forall c \in \CC_g, \label{eq:SCOPF1-alpha-zero}
\\
& \, \,   \delta^c_{g_c} = -p_{g_c}, \quad \forall c \in \CC_g,  \label{eq:SCOPF1-delta-Pg}
\\
& \, \, \sum_{g \in \GG} {\delta^c_g} = 0, \quad \forall c \in \CC, \label{eq:SCOPF1-delta-sum}
\\
& \, \,  \delta^c_g \leq r_g^{\out},   \quad \forall g \in \GG, \forall c \in \CC_g, \label{eq:SCOPF1-delta-bound}
\\
& \, \,  \delta^c_g \geq 0, \quad \forall g \in \GG, g \not= g_c, \forall c \in \CC, \label{eq:SCOPF1-delta-possitive}
\\
& \, \, \delta^c_g = 0, \quad \forall g \in \GG, \forall c \in \CC_\ell, \label{eq:SCOPF1-delta-zero}
\\
& \, \,  \Pb \left\{
    \begin{array}{ll}
        {\alpha_g^c}{\Pt_{\mis}} \leq r_g^{+}, \forall g \in \GG, \\
        -{\alpha_g^c}{\Pt_{\mis}} \leq r_g^{-}, \forall g \in \GG,\\
         \big| \sum_{g \in \GG} M^c_{gl}(p_g + \alpha_g^c  \Pt_{\mis} +\delta^c_g) 
         \\
         \, \, +\sum_{w \in \WW}{M^c_{wl}{\Pt_w}}  + \sum_{d \in \DD}{M^c_{dl}{P_d}} \big|
          \\
          \qquad \leq P_l^{\max}, \forall l \in \LL, \forall c \in \CC
    \end{array}
\right\} \geq 1-\epsilon.
 \label{eq:SCOPF1-cc}
\end{align}
\end{subequations}
The objective function~\eqref{eq:SCOPF1-obj} represents the total operational cost of the system, where the first term stands for cost of planned generation, the second and the third take into account the cost of procuring reserves that handle uncertainty, and the last term quantifies the cost of reserves dedicated for generator contingencies. Constraint~\eqref{eq:SCOPF1-gen-eq} ensures that the planned power generation in the first stage satisfies the load along with the forecast renewable generation levels. Constraints~\eqref{eq:SCOPF1-gen-bound-up} and~\eqref{eq:SCOPF1-gen-bound-down} impose power generation limits under maximum up and down regulation provided by reserves. Bounds on the acquired reserve capacities are given in~\eqref{eq:SCOPF1-res-bound}. Using~\eqref{eq:SCOPF1-alpha-1},~\eqref{eq:SCOPF1-alpha-nn}, and~\eqref{eq:SCOPF1-alpha-zero}, we restrict the control policies that handle uncertainties to a simplex while making sure that the generator under contingency for any $c \in \CC_g$ does not participate in power provision. Constraints~\eqref{eq:SCOPF1-delta-Pg} to~\eqref{eq:SCOPF1-delta-zero} balance the loss in generation due to generator contingencies. Here,~\eqref{eq:SCOPF1-delta-Pg} enforces the post-contingency power to be zero for the failed generator,~\eqref{eq:SCOPF1-delta-sum} guarantees that the power loss is accommodated by other active generators,~\eqref{eq:SCOPF1-delta-bound} ensures that the reserve activation is bounded by the reserve capacity $r_g^{\out}$,~\eqref{eq:SCOPF1-delta-possitive} imposes the change in generation due to contingency not to be negative, and~\eqref{eq:SCOPF1-delta-zero} makes reserve activation to be zero when there are only line contingencies. Finally,~\eqref{eq:SCOPF1-cc} is a chance constraint containing three sets of inequalities which must be satisfied jointly with high probability $1-\epsilon$. The first two sets of constraints within the chance constraint are for upper and lower reserves assigned to renewable uncertainty. The third constraint implies that in the presence of renewable uncertainty the line flow must satisfy line limits, where $M^c_{gl}, M^c_{wl}, M^c_{dl}$ are Power Transfer Distribution Factors (PTDFs) which respectively translate conventional power injections, 
renewable energy injections, and load demand into line flows. These matrices can be computed, for example, by implementing the method outlined in~\cite{MV-KM-JL-GA:13}. In the chance constraint the parameter $P_l^{\max}$ denotes the allowed power flow capacity of line $l \in \LL$. 

Our aim in this paper is to solve the above chance-constrained optimization problem using data. To this end, we replace the chance-constraint in problem~\eqref{eq:SCOPF1} with its distributionally robust (DR) counterpart and derive a computationally efficient method for approximating the solution of the resulting optimization problem.

\section{Two-step distributionally robust method for solving CCSCD}\label{sec:method}
For the sake of convenience, we rewrite the CCSCD problem~\eqref{eq:SCOPF1} in the following compact form:
\begin{subequations}\label{eq:simple}
\begin{align}
\min_x & \quad c^{\top}x \label{eq:simple-obj}
\\
\st &  \quad \Lambda x  \leq \beta,  \label{eq:simple-norandom}
\\
&   \quad \Pb (e_k^{\top}x+f_k^{\top} \uncertain + x^{\top}F_k \uncertain  \leq h_k, \forall k \in \until{K}) \! \ge \! 1 \! - \! \epsilon, \label{eq:simple-random}
\end{align}
\end{subequations}
where $x \in \real^{n_x}$ stands for the decision variables and $\uncertain \in \real^{n_\uncertain}$ denotes the uncertain generation in the network given by the vector $(\Pt_w)_{w \in \WW}$. We use the notation $[K] := \{1,2,\dots,K\}$. Recall that $\uncertain$ has the distribution $\mathbb{P}$ and support $\Xi$. The constraint~\eqref{eq:simple-norandom} captures the deterministic constraints~\eqref{eq:SCOPF1-gen-eq} to~\eqref{eq:SCOPF1-delta-zero} in the formulation~\eqref{eq:SCOPF1}. It is worth mentioning that deterministic equality constraints in~\eqref{eq:SCOPF1} are treated using two inequality constraints in~\eqref{eq:simple-norandom}. Note that here, $K = 2\abs{\CC}(\abs{\GG} + \abs{\LL})$ is the number of constraints included in the joint chance-constraint \eqref{eq:SCOPF1-cc} and vectors $e_k$, $f_k$, $h_k$ and matrix $F_k$ represent in a general form the $k$-th constraint.

Motivated by the real-life situation, we assume that the distribution $\Pb$ is not entirely known and instead only a finite number of samples of $\uncertain$ are available. Our approach then is to approximate the solution of~\eqref{eq:simple} with a DR optimization problem using the available data.
To this end, let $\Xihat^N := \{\uncertainhat_1, \uncertainhat_2, \dots, \uncertainhat_N\}$ be the set of $N$ samples of the uncertainty and let $\Qbhat_N := \frac{1}{N} \sum_{m=1}^N \delta_{\uncertainhat_m}$ be the empirical distribution, where $\delta_{\uncertainhat_m}$ is the delta function at the point $\uncertainhat_m$. Given a radius $\theta \ge 0$ and the dataset, the Wasserstein ambiguity set is given by 
\begin{align}\label{eqn:ambiguityset}
\MM_N^{\theta} := \{{\Qb} \in \mathcal{P}(\Xi) \mid d_w(\Qb,\Qbhat_N) \leq \theta \},
\end{align}
where $\mathcal{P}(\Xi)$ is space of all distributions  supported on $\Xi$ and $d_w$ stands for the Wasserstein metric. The ambiguity set contains all distributions that are at most $\theta$ away from the empirical distribution in the Wasserstein metric~\cite{PME-DK:18} $d_w: \mathcal{P}(\Xi)\times\mathcal{P}(\Xi) \to \real$ defined as: 
\begin{align}\label{eqn:W-metric}
    d_w(\Qb,\Qbhat_N) := \min_{\gamma \in \mathcal{H}(\Qb,\Qbhat_N)} \int_{\Xi\times\Xi} \norm{\uncertain -  \uncertainhat} \gamma(d\uncertain,d\uncertainhat),
\end{align}
where $\norm{\cdot}$ represents a norm in the Euclidean space and $\mathcal{H}(\Qb,\Qbhat_N)$ is the set of distributions on the space $\Xi \times \Xi$ with marginals $\Qb$ and $\Qbhat_N$. 
The DR version of~\eqref{eq:simple} based on the ambiguity set $\MM_N^{\theta}$ is given as:
\begin{subequations}\label{eq:DRCCP}
\begin{align}
\min_x & \quad c^{\top}x \label{eq:DRCCP-obj}
\\
\st &  \quad \Lambda x  \leq \beta,  \label{eq:DRCCP-norandom}
\\
&   \quad \inf_{\Qb \in \MM_N^{\theta}}\Qb (e_k^{\top}x+f_k^{\top}\uncertain +x^{\top}F_k\uncertain  \leq h_k, \forall k \in \until{K}) \notag
\\
&  \qquad \qquad \qquad \qquad \qquad \qquad \qquad \qquad\ge 1-\epsilon. \label{eq:DRCCP-random}
\end{align}
\end{subequations}
The above problem requires that each joint chance constraint holds for all distributions in the ambiguity set. The constraint~\eqref{eq:DRCCP-random} contains an infinite-dimensional optimization problem. However, it can be reformulated exactly into a mixed-integer linear program~\cite{ZC-DK-WW-18,WX:21} or a bilinear optimization problem~\cite{AA-CO-JK-ZD-JT-FV:22} under fairly mild assumptions. Hence, the resulting optimization problem is of large scale and computationally burdensome. In particular, the number of bilinear constraints is of the order of the product of number of samples, size of the network, and the number of contingencies. Thus, due to this coupling, the problem becomes difficult to solve when the network and the number of samples are large. To overcome this limitation, we propose the following two-step framework:
\noindent \emph{Step 1:} Construct a polyhedral set $\Xirob \subset \Xi$  such that 
    \begin{align}\label{eq:xirob-guarantee}
        \inf_{\Qb \in \MM_N^\theta}  \Qb(\uncertain \in \Xirob ) \ge 1-\epsilon.
    \end{align}
    
\noindent \emph{Step 2:} Solve the robust optimization problem
    \begin{equation}\label{eq:robust-gen}
    \begin{aligned}
        \min_x & \quad c^{\top}x 
        \\
        \st &  \quad \Lambda x  \leq \beta,  
        \\
        &   \quad e_k^{\top}x+f_k^{\top}\uncertain + x^{\top}F_k\uncertain  \leq h_k,  \forall  \xi \in \Xirob, k \in \until{K}. 
    \end{aligned}
    \end{equation}

Note that due to the property of $\Xirob$ obtained in Step 1 of the above framework, any optimal solution of the robust problem~\eqref{eq:robust-gen} is a feasible solution of the DR problem~\eqref{eq:DRCCP}. We will establish this property later. How close  the optimizer is to the optimal solution of~\eqref{eq:DRCCP} depends on the ability of finding the ``most relevant'' set that satisfies the condition~\eqref{eq:xirob-guarantee}. Moreover, the characterization of $\Xirob$, that of being defined by linear or quadratic constraints for example, also determines the computational complexity of the two steps. Thus, the choice of $\Xirob$ inherently involves the trade-off between the accuracy and computational performance of the method. Notice that the complexity of the first step is independent of the size of the network and the number of contingencies, it only depends on the number of samples and the size of the uncertainty. On the other hand, the second step is agnostic of the number of data points. This proves to be computationally advantageous, as will be discussed further in the numerical example.

In~\cite{BP-AH-SB-DC-AC:20}, the same two-step approach was proposed but the $\Xirob$ set was restricted to be a hyper-rectangle. This choice was conservative as it could not take into account the correlation present between various renewables. Next, we provide one way of determining $\Xirob$ which overcomes this limitation. Later, we focus on the tractability of the robust problem.

\subsection{Constructing $\Xirob$}

We assume that the uncertain generation across the network is correlated in a sparse manner, i.e., renewable generators can be divided into subsets and the correlation exists only among generators belonging to one subset. To make this precise, let $\WW = \{w_1, w_2, \dots, w_{n_\uncertain}\}$ be the set of all renewable generators. We partition these into $n_W$ number of sets $\{\WW_1, \WW_2, \dots, \WW_{n_W}\}$ such that $\WW_i \subset \WW$ for all $i \in [n_W]$,  $\WW_i \cap \WW_j = \emptyset$ for all $i, j \in [n_W]$, and $\bigcup_{i=1}^{n_W} \WW_i =  \WW$. We assume that uncertain generations of two generators are correlated only if they belong to one partition. Consequently, we assume that the support of every distribution belonging to the ambiguity set $\MM_N^\theta$ is given as 
\begin{align}\label{eq:support-form}
    \Xi = \prod_{i\in \until{n_W}} \Xi^i, \text{ with } \Xi^i = \setdef{\xi^i \in \real^{\abs{\WW_i}}}{\Gamma^i \xi^i \le \rho^i},
\end{align}
where $\Gamma^i \in \real^{u_i \times \abs{\WW_i}}$ and $\rho^i \in \real^{u_i}$ define the support of the marginal of the distribution over the partition $\WW_i$. Due to correlations restricted to a partition,  we aim to find an appropriate polyhedral set $\Xirob^i \subset \real^{\abs{\WW_i}}$ for each of the partitions by solving a DR problem. Then, we construct $\Xirob := \prod_{i=1}^{n_W} \Xirob^i$. This procedure is summarized in Algorithm~\ref{alg:c-xirob}. Its detailed description is given below.

\begin{quote}
    \emph{[Informal description of Algorithm~\ref{alg:c-xirob}]:} The procedure admits as input the dataset $\Xihat^N$, violation level $\epsilon$, radius $\theta$, and partition $\{\WW_1, \WW_2, \dots, \WW_{n_W}\}$. For each sample $\uncertainhat_m$, the components corresponding to the partition $\WW_i$ are collected in the vector $(\uncertainhat_m)^i$. The algorithm starts with computing the empirical covariance matrix $\widehat{\mathrm{Cov}}^i \in \real^{\abs{\WW_i}\times\abs{\WW_i}}$ for each partition using the samples (Line~\ref{ln:cov}). Next, eigenvalues $\{\lambda_{(i,1)}, \dots, \lambda_{(i,\abs{\WW_i})}\}$ and eigenvectors $\{v_{(i,1)}, \dots, v_{(i,\abs{\WW_i})}\}$ of $\widehat{\mathrm{Cov}}^i$ are computed (Line~\ref{ln:compute}), where without loss of generality, we assume that the eigenvalues are arranged in an ascending manner. Our idea of constructing $\Xirob^i$ involves intersecting a number of \emph{simple} sets, each consisting of a set of points lying between two hyperplanes that share a common normal. These simple sets can be divided into two groups, one where the normal is aligned with one of the axis of the Euclidean space where $\Xirob^i$ lies, i.e., $\real^{\abs{\WW_i}}$ and the other with the normal being the eigenvectors of $\widehat{\mathrm{Cov}}^i$ with the smallest eigenvalues. To this end,  we select a number $0 \le \kappa_i < \abs{\WW_i}$ determining the number of eigenvector-based simple sets to be constructed. An illustration of this construction in two dimensions is given in Figure \ref{fig:propsed-illustration}.
    We set $n^{\mathrm{DR}}_i := \kappa_i + \abs{\WW_i}$ as the number of sets that will be intersected to determine $\Xirob^i$. 
     Specifically, we set
     \begin{align}\label{eq:xirobi}
         \Xirob^i := \bigcap_{j=1}^{n^{\mathrm{DR}}_i} \Xirob^{i,j},
     \end{align}
     where each set $\Xirob^{i,j}$ is of the form
    \begin{align}\label{eq:xirobij}
         \Xirob^{i,j} := \setdefb{\uncertain^i \in \real^{\abs{\WW_i}}}{\underline{b}_{i,j} \le \zeta_{i,j}^\top \uncertain^i \le \overline{b}_{i,j}}.
    \end{align}
    For the above set, the vector $\zeta_{i,j} \in \real^{\abs{\WW_i}}$ determines the normal to the set of hyperplanes $\zeta_{i,j}^\top \uncertain^i = \underline{b}_{i,j}$ and $\zeta_{i,j}^\top \uncertain^i = \overline{b}_{i,j}$ that construct $\Xirob^{i,j}$. If $1\le j \le \abs{\WW_i}$, then we set $\zeta_{i,j} \in \real^{\abs{\WW_i}}$ as the unit vector along the $j$-th coordinate, otherwise $\zeta_{i,j} = v_{(i,j-\abs{\WW_i})}$ is the eigenvector of the covariance matrix corresponding to the eigenvalue $\lambda_{(i,j-\abs{\WW_i})}$. The bounds $\underline{b}_{i,j},\overline{b}_{i,j}$ are obtained as the solutions of the DR optimization problem~\eqref{eq:parallel}. Finally, we set $\Xirob = \prod_{i=1}^{n_W} \Xirob^i$ as the Cartesian product of the polyhedral sets obtained for each partition. We denote $\ndr := \sum_{i=1}^{n_W} \ndr_i$. 
    \oprocend
\end{quote}

\begin{figure}[h]
\centering
\includegraphics[scale=0.5]{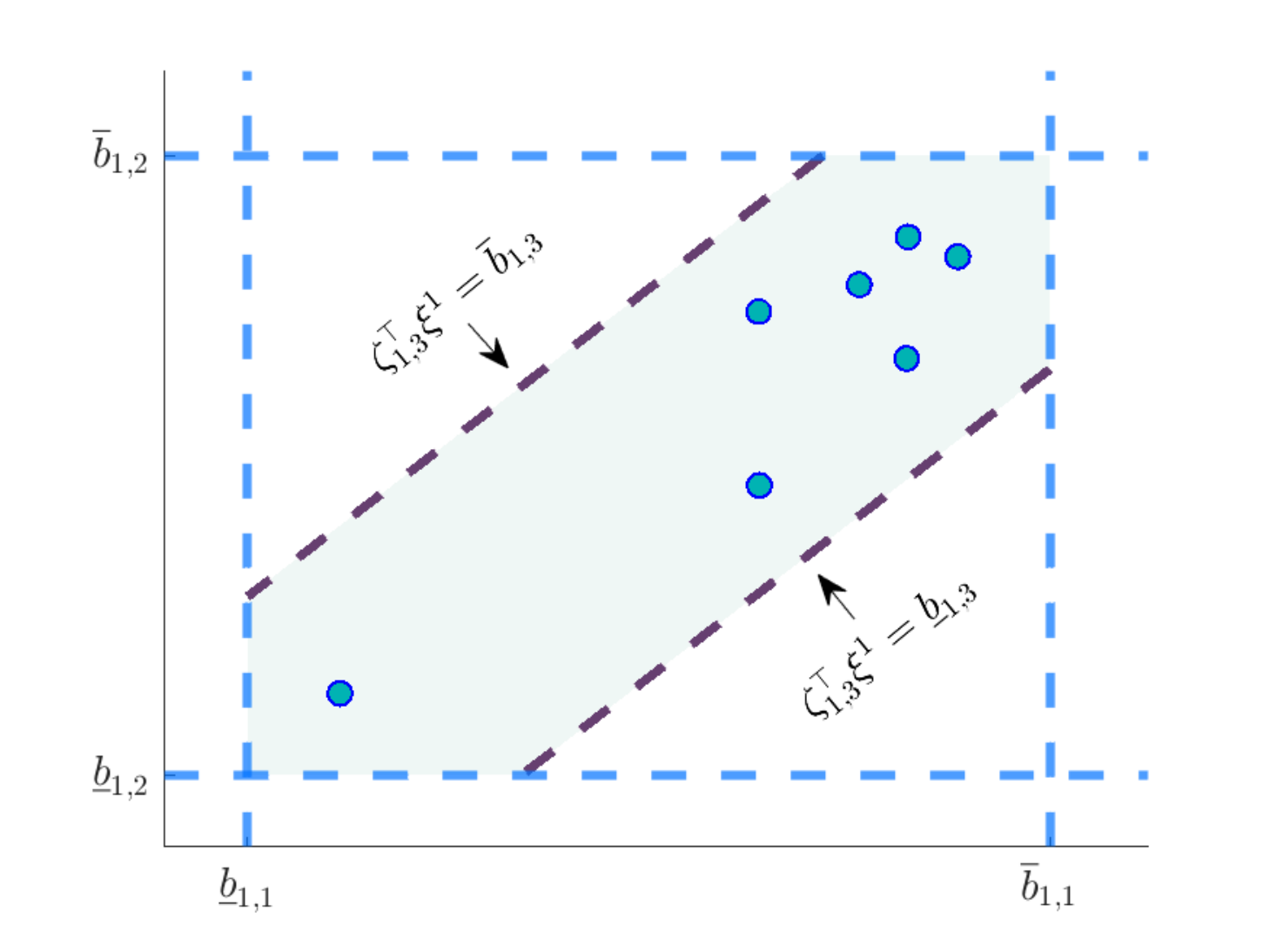}
\caption{Illustration of constructing $\Xirob$. This example contains two correlated random variables that form one partition, i.e., $n_W=1$ and $\abs{\WW_1}=2$. For this partition, the number of simple sets is $n^{\mathrm{DR}}_1 := 3$, where we select $\kappa_1 = 1$. Thus, $\Xirob$ is constructed as the intersection of three simple sets, i.e., $\Xirob = \Xirob^1 \cap \Xirob^2 \cap \Xirob^3$. We are given seven samples depicted by blue dots in the plot. The set $\Xirob^1$ is the set of points that lie between the two vertical lines. That is, for this set the normal is aligned along the horizontal axis. Analogously, the set $\Xirob^2$ is the set of points that lie between the two horizontal lines forming the rectangle with $\Xirob^1$ altogether. The set of all points that lie between the inclined lines is $\Xirob^3$. This set is unbounded but not depicted that way for clarity. Finally, $\Xirob$ is the intersection of all $\Xirob^j$, $j=1,\dots,3$ which is represented as a green shaded region. 
}
\label{fig:propsed-illustration}
\end{figure}
\begin{algorithm}
    \SetAlgoLined
	\DontPrintSemicolon
	\SetKwInOut{Input}{Input}
    \SetKwInOut{Output}{Output}
	\SetKwInOut{init}{Initialize}
	\SetKwInOut{giv}{Data}
	\giv{ Data $\Xihat = \{\uncertainhat^1, \uncertainhat^2, \dots, \uncertainhat^N\}$, partition $\{\WW_1, \WW_2, \dots, \WW_{n_W}\}$, Wasserstein radius $\theta$, and tolerance $\epsilon$
}
    \For{$i \in \until{n_W}$}{
    Compute the empirical covariance matrix $\widehat{\mathrm{Cov}}^i$ for each partition $\WW_i$ \label{ln:cov} \;
    Compute eigenvectors and eigenvalues of $\widehat{\mathrm{Cov}}^i$ \label{ln:compute}\;
	Select $\kappa_i \in \{0,\dots,\abs{\WW_i}-1\}$ to be the number of eigenvectors to be used to find hyperplane sets \;
	Set $n^{\mathrm{DR}}_i := \kappa_i + \abs{\WW_i}$\;
	}

	\For{$i \in \until{n_W}$}{
	Construct $\Xirob^i$ as~\eqref{eq:xirobi} using definition~\eqref{eq:xirobij} and solving~\eqref{eq:parallel} \;
	}
	Construct $\Xirob = \prod_{i=1}^{n_W} \Xirob^i$\;
	\caption{Constructing $\Xirob$}
	\label{alg:c-xirob}
\end{algorithm}

The key ingredient of Algorithm~\ref{alg:c-xirob} is the DR optimization problem explained next that gives the half-spaces defining each $\Xirob^{i,j}$. Given a partition $i \in \setr{1,\dots,n_W}$ and an normal vector $\zeta_{i,j}$, consider the following DR problem  
\begin{subequations}\label{eq:parallel}
    \begin{align}
        \min_{(\underline{b},\overline{b})} & \quad \overline{b} - \underline{b}
        \\
        \st &  \quad \underline{b} \le \overline{b} ,
        \\
        &  \quad \inf_{\Qb \in \MM_N^\theta} \Qb \Bigl( \max \{\zeta_{i,j}^{\top} \uncertain^i - \overline{b}, \underline{b}- \zeta_{i,j}^{\top}\uncertain^i\} \le 0 \Bigr) \notag
        \\
        & \qquad \qquad \qquad \ge 1- \frac{\epsilon}{\ndr}. \label{eq:parallel-sup}
\end{align}
\end{subequations}
In the above problem, the constraint~\eqref{eq:parallel-sup} enforces that the probability mass placed on the set $\setdef{\uncertain^i}{\underline{b} \le \zeta_{i,j}^\top \uncertain^i \le \overline{b}}$ for any feasible point $(\underline{b},\overline{b})$ is no less than $1 - \epsilon/ \ndr$ under any distribution in $\MM_N^\theta$. This constraint eventually ensures that we identify the right set $\Xirob$. The above problem is difficult to solve in practice as is. Below we provide a reformulation that helps in computation. Subsequently, we provide a heuristic to find an approximate optimizer of the problem.
\begin{proposition} \longthmtitle{Finite-dimensional reformulation of~\eqref{eq:parallel}}\label{pr:tractable}
    Assume that the support of each distribution in $\MM_N^\theta$ is of the form~\eqref{eq:support-form}. Consider problem~\eqref{eq:parallel} written for some partition $i \in \until{n_W}$ and simple set $j \in \until{ n_i^{\mathrm{DR}}}$. 
    Given a normal $\zeta \in \real^{\abs{\WW_i}}$ and a constant $b \in \real$,  we define the notation
    \begin{align}\label{eq:HH_i}
        \mathcal{H}_i(\zeta,b) := \Xi^i \cap \setdef{\xi^i}{\zeta^\top \xi^i = b}.
    \end{align}
    as the intersection of $\Xi^i$ and the hyperplane formed with $(\zeta,b)$. Then, the optimization problem~\eqref{eq:parallel} is equivalent to 
     \begin{subequations}\label{eqn:DRBounds2}
         \begin{align}
             \min & \quad \overline{b} - \underline{b}
             \\
             \wrt & \quad {\underline{b},\overline{b},\lambda \ge 0, \{s_m \ge 0\}_{m=1}^N}
             \\
             \st & \quad   \underline{b} \leq \overline{b},
             \\
             & \quad \lambda \theta + \frac{1}{N}  \sum_{m=1}^{N} s_m \leq \frac{\epsilon}{\ndr}, \label{eqn:DRBounds2_sum}
             \\
             & \quad \text{ for all } m \in \until{N}: \notag
             \\
             & \quad \qquad  {s_m}  \geq 1  -  \lambda \dist\Bigl((\uncertainhat_m)^i,\mathcal{H}_i(\zeta_{i,j},\overline{b})\Bigr), \label{eqn:DRBounds2_d1}
             \\
             & \quad \qquad  {s_m} \geq 1- \lambda \dist\Bigl((\uncertainhat_m)^i,\mathcal{H}_i(\zeta_{i,j},\underline{b})\Bigr),  \label{eqn:DRBounds2_d2}
         \end{align}
    \end{subequations}
    where $\dist\Bigl((\uncertainhat_m)^i,\mathcal{H}_i(\zeta_{i,j},\overline{b})\Bigr)$ is the distance between the sample $(\uncertainhat_m)^i$ and the set $\mathcal{H}_i(\zeta_{i,j},\overline{b})$ under the norm $\norm{\cdot}$ and the notation 
    $\dist\Bigl((\uncertainhat_m)^i,\mathcal{H}_i(\zeta_{i,j},\underline{b})\Bigr)$ is analogous. These distances can be computed by solving the following optimization problem
    \begin{align} \label{eq:dist}
    & \dist\Bigl((\uncertainhat_m)^i,\mathcal{H}_i(\zeta_{i,j},b)\Bigr) \notag
    \\
    & \qquad = \begin{cases} \max\limits_{y,z} &  (b - \zeta^{\top}_{i,j} (\uncertainhat_m)^i)y - (\rho^i - \Gamma^i  (\uncertainhat_m)^i)^{\top} z %\label{eq:dist_obj}
    \\
    \st &  \norm{\zeta y - (\Gamma^i)^{\top} z}_* \leq 1,
    \\
    &  y \ge 0, z \in \realnonnegative^{\abs{\WW_i}},
    \end{cases}
    \end{align}  
    where $\norm{\cdot}_*$ is the dual norm of $\norm{\cdot}$ used in~\eqref{eqn:W-metric}.
\end{proposition}

\begin{proof}
    The key step involves reformulating~\eqref{eq:parallel-sup}. Recall that for a point $\uncertain \in \real^{n_\uncertain}$, we use the notation $\uncertain = (\uncertain^1, \uncertain^2, \dots, \uncertain^{n_W})$ to represent the components belonging to each of the partitions. That is, $\uncertain^i \in \real^{\abs{\WW_i}}$ consists of  components corresponding to the partition $i$. Denote $\bt = (\bu,\bo)$ and  we consider  $Z(\bt,\uncertain) := \max \{{\zeta_{i,j}}^{\top} \uncertain^i - \bo, \bu- {\zeta_{i,j}}^{\top}\uncertain^i \}$. Then, we compute
\begin{align*}
    & \sup_{\Qb \in \MM_\theta^N} \Qb ( Z(\bt,\uncertain) > 0)  = \sup_{\Qb \in \MM_\theta^N} \Eb_\Qb [\mathbbm{1}_{\mathrm{cl} \setdef{\uncertain \in \Xi}{Z(\bt,\uncertain) > 0}}]
\end{align*} 
where the equality follows from \cite[Proposition 4]{RG-AK:16}. Following from \cite[Proposition 1]{RG-AK:16}, it is equivalant to
\begin{align*}
    \inf_{\lambda \geq 0, s} & \quad \lambda\theta +  \dfrac{1}{N}\sum_{m=1}^{N}{s_m} 
    \\
    \st & \quad  {s_m} \geq \sup_{\uncertain \in \Xi} \bigl[{\mathbbm{1}}_{\mathrm{cl} \setdef{\uncertain \in \Xi}{Z(\bt,\uncertain) > 0}} - \lambda \norm{\uncertain - \uncertainhat_m}\bigr].
\end{align*} 

Let us define $\Xi_1 := \mathrm{cl} \setdef{\uncertain \in \Xi}{Z(\bt,\uncertain) > 0}$ and $\Xi_2 = \Xi \backslash \Xi_1$. Then, we have
\begin{align} \label{eq:sm-ineq-1}
    {s_m} \geq \max \Bigl\{\sup_{\uncertain \in \Xi_1}1 - \lambda{\norm{\uncertain - \uncertainhat_m}} , \sup_{\uncertain \in \Xi_2} - \lambda{\norm{\uncertain - \uncertainhat_m}} \Bigr\}.
\end{align}
Note that since the function $Z$ only depends on the component $\uncertain^i$ of the vector $\uncertain$, we can write 
\begin{subequations}\label{eq:xi-split}
\begin{align*}
    \Xi_1 & = \prod_{k=1}^{i-1} \Xi^k \times \Xi^i_1 \times \prod_{k=i+1}^{n_W} \Xi^k,
    \\
    \Xi_2 & = \prod_{k=1}^{i-1} \Xi^k \times (\Xi^i \setminus \Xi^i_1) \times \prod_{k=i+1}^{n_W} \Xi^k,
\end{align*}
\end{subequations}
where we use the notation
\begin{align*}
    \Xi_1^i := \mathrm{cl} \setdef{\xi^i \in \Xi^i}{Z(\bt,\xi^i) > 0},
\end{align*}
with a slight abuse with the fact that now $Z(\bt,\xi^i) = \max \{{\zeta_{i,j}}^{\top} \uncertain^i - \bo, \bu- {\zeta_{i,j}}^{\top}\uncertain^i \}$. Using the decomposition~\eqref{eq:xi-split} in~\eqref{eq:sm-ineq-1}, we get
\begin{align}\label{eq:sm-ineq-2}
    {s_m} \geq \max \Bigl\{\sup_{\uncertain^i \in \Xi_1^i}1 - \lambda{\norm{\uncertain^i - (\uncertainhat_m)^i}} \, \, , \, \, \sup_{\uncertain^i \in \Xi^i_2} - \lambda{\norm{\uncertain^i - (\uncertainhat_m)^i}} \Bigr\}.
\end{align}
To handle the terms inside the $\max$ operator, note that if $(\uncertainhat^m)^i \in \Xi^i_1$ the first term is 1 while the second term is non-positive, hence ${s_m} = 1$. Similarly, if $(\uncertainhat^m)^i \in \Xi^i_2$, the second term is zero. Thus, we can rewrite~\eqref{eq:sm-ineq-2} as 
%\begin{align*}
    ${s_m} \geq \max \{0,1- \inf_{\uncertain^i \in \Xi^i_1} {\lambda}\norm{\uncertain^i - (\uncertainhat_m)^i}\}$,
%\end{align*}
which is equivalent to:
\begin{align}\label{eq:sm-ineq-3}
{s_m} \geq 0, {s_m} \geq 1- \inf_{\uncertain^i \in \Xi^i_1}{\lambda}\norm{\uncertain^i - (\uncertainhat_m)^i}.
\end{align}
Following the definition of $\Xi^i_1$, we have
\begin{align*}
    & \inf_{\uncertain^i \in \Xi^i_1} {\lambda}\norm{\uncertain^i - (\uncertainhat_m)^i} 
    \\
    & = \min \Bigl\{\dist((\uncertainhat_m)^i,\mathcal{H}_i(\zeta_{i,j},\overline{b})),\dist((\uncertainhat_m)^i,\mathcal{H}_i(\zeta_{i,j},\underline{b})) \Bigr\},
\end{align*}
where $\mathcal{H}_i$ is given in~\eqref{eq:HH_i}. Using the above relation, we write~\eqref{eq:sm-ineq-3} equivalently as $s_m \ge 0$, ${s_m} \geq 1- \lambda \dist((\uncertainhat_m)^i,\mathcal{H}_i(\zeta_{i,j},\overline{b}))$, and ${s_m} \geq 1- \lambda \dist((\uncertainhat_m)^i,\mathcal{H}_i(\zeta_{i,j},\underline{b}))$. This shows the equivalence between~\eqref{eq:parallel} and~\eqref{eqn:DRBounds2}. The expression for  $\dist((\uncertainhat_m)^i,\mathcal{H}_i(\zeta,b))$ for a given $\zeta$ as given in~\eqref{eq:dist} is established in~\cite[Lemma 1]{AA-CO-JK-ZD-JT-FV:22}. This completes the proof.
\end{proof} 

\begin{remark}\longthmtitle{Approximating the solution of~\eqref{eq:parallel} using Proposition~\ref{pr:tractable}}
    While Proposition~\ref{pr:tractable} provides a finite-dimensional exact reformulation of~\eqref{eq:parallel}, the resulting  optimization~\eqref{eqn:DRBounds2} is difficult to solve as~\eqref{eqn:DRBounds2_d1} and~\eqref{eqn:DRBounds2_d2} are bilinear in nature if one uses the formulation~\eqref{eq:dist}. Therefore, we provide a heuristic way of solving~\eqref{eqn:DRBounds2}. We implement an iterative procedure. We start with some values for $\underline{b}$ and $\overline{b}$ and check for the feasibility of constraints~\eqref{eqn:DRBounds2_sum} to~\eqref{eqn:DRBounds2_d2}. Since  $\underline{b}$ and $\overline{b}$ are fixed, all distances in~\eqref{eqn:DRBounds2_d1} and~\eqref{eqn:DRBounds2_d2} can be computed by solving the convex program~\eqref{eq:dist}. If the feasibility is satisfied, then we shrink the interval $[\underline{b},\overline{b}]$ by a factor and check the feasibility again. We iterate this process till we find the shortest possible interval. 
    
    Note that instead of appealing to the above heuristic, one can use~\eqref{eq:dist} in constraints~\eqref{eqn:DRBounds2_d1} and~\eqref{eqn:DRBounds2_d2} to form a bilinear optimization problem and solve it as is using methods tailor-made for bilinear problems. In the interest of space, we do not provide further details regarding it. 
\oprocend
\end{remark}

The following result summarizes the property of $\Xirob$. %as formed using Algorithm~\ref{alg:c-xirob}.
\begin{proposition}\longthmtitle{Guarantees for $\Xirob$}\label{le:prop-xirob}
    Let $\Xirob$ be the output of Algorithm~\ref{alg:c-xirob}. Then, we have
    \begin{align}\label{eq:Xirob-dr-bound}
        \inf_{\Qb \in \MM^\theta_N} \Qb(\uncertain \in \Xirob) \ge 1-\epsilon.
    \end{align}
\end{proposition}
\begin{proof}
   Considering $\Xirob = \bigcap_{i=1}^{n_W}\bigcap_{j=1}^{\ndr_i} \Xirob^{i,j}$ and the fact that the complement of the intersection of sets is the same as the union of their complements, we obtain the bound:
   \begin{align*}
    \inf_{\Qb \in \MM^\theta_N} \Qb(\uncertain \in \Xirob) &= 1-\sup_{\Qb \in \MM^\theta_N} \Qb(\uncertain \notin \Xirob)
    \\
    & =  1 - \sup_{\Qb \in \MM^\theta_N} \Qb(\cup_{i=1}^{n_W}\cup_{j=1}^{\ndr_i}  \{\uncertain \notin \Xirob^{i,j}\})
    \\
    & \stackrel{\text{(a)}}{\ge} 1-  \sum_{i=1}^{n_W} \sum_{j=1}^{\ndr_i} \sup_{\Qb \in \MM^\theta_N} \Qb(\uncertain \notin \Xirob^{i,j})
    \\
    & \stackrel{\text{(b)}}{\ge} 1 -  \sum_{i=1}^{n_W} \sum_{j=1}^{\ndr_i} \frac{\epsilon}{\ndr}  = 1-\epsilon.
    \end{align*}
    In the above expression, (a) follows from the union bound, and (b) is due to \eqref{eq:parallel-sup}.
\end{proof}

\subsection{Robust optimization using $\Xirob$}

Having constructed the set $\Xirob$ in the previous section, we propose solving the following robust optimization problem to approximate the solution of the DR problem~\eqref{eq:DRCCP}:
\begin{subequations}\label{eq:poly}
\begin{align}
\min_x & \, \, \,  c^{\top}x 
\\
\st &  \, \, \, \Lambda x  \leq \beta,  
\\
&   \, \, \, e_k^{\top}x+f_k^{\top}\xi+x^{\top}F_k\xi  \leq h_k, \, \forall \xi \in \Xirob, k \in \until{K}.  \label{eq:polyrob}
\end{align}
\end{subequations}
Note that the set $\Xirob$ as constructed in Algorithm~\ref{alg:c-xirob} is polyhedral and it can be written in a general form as
\begin{align}\label{eq:polyset}
  \Xirob = \setdef{ \uncertain \in \real^{n_\uncertain} }{ G \xi \leq g },
\end{align}
where $G \in \real^{q \times n_{\uncertain}}$ and $g \in \real^{q}$, with $q = 2 \ndr$, are obtained by Algorithm~\ref{alg:c-xirob}. In order to solve the robust problem~\eqref{eq:poly}, the following result provides a reformulation where the robust constraint is replaced with a set of affine ones. 

\begin{proposition}\longthmtitle{Reformulation of~\eqref{eq:poly} with polyhedral uncertainty set}
The optimization problem~\eqref{eq:poly} with the uncertainty set given in~\eqref{eq:polyset} is equivalent to the following problem
\begin{subequations}\label{eq:poly2}
\begin{align}
\min_{x,\setr{y_k}_{k \in \until{K}}} & \quad c^{\top}x 
\\
\st &  \quad \Lambda x  \leq \beta,  
\\
&   \quad g^{\top}y_k + e_k^{\top}x  \leq h_k , & \quad \forall k \in \until{K},
\\
&  \quad F_k^{\top}x - G^{\top}y_k + f_k = 0, & \quad \forall k \in \until{K},
\\
&  \quad y_k \geq 0, & \quad \forall k \in \until{K}.
\end{align}
\end{subequations}
Here, equivalence implies that $x^*$ is an optimal solution of~\eqref{eq:poly} if and only if there exist $\setr{y_k^*}_{k \in \until{K}}$ such that $(x^*,\setr{y_k^*}_{k \in \until{K}})$ is an optimal solution of~\eqref{eq:poly2}.
\end{proposition}
\begin{proof}
The program \eqref{eq:poly} reads as:
\begin{subequations}\label{eq:poly2}
\begin{align}
\min & \quad c^{\top}x 
\\
\st &  \quad \Lambda x  \leq \beta,  
\\
&   \quad \max_{\uncertain \in \Xirob} (f_k^{\top}+x^{\top}F_k)\uncertain  \leq h_k - e_k^{\top}x,  \quad \forall k \in \until{k}. \label{eq:poly2_max}
\end{align}
\end{subequations}
We use duality to reformulate the semi-infinite constraint~\eqref{eq:poly2_max} in the above problem. Note the following primal-dual pair of optimization problems:
\begin{align*}
    \max_{\uncertain \in \Xirob} (f_k^{\top}+x^{\top}F_k)\uncertain = \begin{cases} \underset{y_k \geq 0}{\min } & \quad y_k^\top g 
    \\
    \st & \quad F_k^{\top}x - G^{\top}y_k + f_k = 0.
    \end{cases}
\end{align*}
Using this equivalence in~\eqref{eq:poly2_max} yields the conclusion. 
\end{proof}
The following result summarizes the properties of our proposed method.
\begin{proposition} \longthmtitle{Approximating~\eqref{eq:DRCCP} with~\eqref{eq:poly}} \label{prop:feas}
    Let $x^*$ be an optimal solution of problem~\eqref{eq:poly}, where $\Xirob$ is constructed using Algorithm~\ref{alg:c-xirob}. Then, $x^*$ is a feasible solution of the DR problem~\eqref{eq:DRCCP}.
\end{proposition}
\begin{proof}
    Note that from Lemma~\ref{le:prop-xirob}, the set $\Xirob$ has the property~\eqref{eq:Xirob-dr-bound}. Trivially, $x^*$ satisfies~\eqref{eq:polyrob} for all $\xi \in \Xirob$. Thus, constraint~\eqref{eq:DRCCP-random} is satisfied, which means $x^*$ is a feasible solution of the DR problem~\eqref{eq:DRCCP}.
\end{proof}

As a consequence of the above result, if the true probability distribution $\Pb$ belongs to the ambiguity set $\MM_N^{\theta}$ with at least probability $p$ for some $p \in (0,1)$, as is the case for i.i.d samples due to bound given in~\cite{PME-DK:18}, then $x^*$ satisfies the joint chance-constraint~\eqref{eq:SCOPF1-cc} and so is feasible for problem~\eqref{eq:SCOPF1} with at least probability $p$. Moreover, this probability can be tuned by adjusting the radius of the ambiguity set based on the number of samples.

\section{Numerical results}\label{sec:sims}
We test our approach on a modified version of the IEEE~RTS~24-bus reference case system~\cite{CO-PP-MG-JM-MZ:16} with 12 conventional generation units 
(see Figure~\ref{fig:24bus}). Note that at node 15 and 23 two generators are located but for the sake of simplicity they are not depicted separately. The network also includes six wind generation units with each having capacity of 200~MW. The total demand in the system is 2900~MW which is distributed across 17 locations with the distribution factors being same as those of the first time interval in the test case~\cite{CO-PP-MG-JM-MZ:16}. The system operator aims to satisfy the demand using conventional and wind generation units in a cost-optimal manner while satisfying (with high probability) all capacity and flow constraints. That is, we wish to solve problem~\eqref{eq:SCOPF1}. We consider 34 contingencies including all possible line and generator contingencies except the contingency of the line that connects bus 7 and 8 as it disconnects the network. The cost coefficients, the minimum and maximum generation capacities, and maximum reserve capacities of the conventional generators are given in Table~\ref{tab:gen-par}. The reactances and capacities of power lines ($P_l^{\max})$ are taken as the values given for the standard IEEE 24 bus system~~\cite{CO-PP-MG-JM-MZ:16}. The PTDF matrices ($M^c_{gl}, M^c_{wl}, M^c_{dl}$) are  calculated based on~\cite{RC-BW-IW:00}. We aim to achieve less than 5\% violation of the power flow and reserve capacity constraints, so we set $\epsilon = 0.05$. 
In this case wind generators are divided into two different partitions. Wind generators located at bus 3, 5, and 7 belong to one partition, and wind  generators located at bus 16, 21, and 23 are in the other partition. This partitioning strategy takes into account the correlation among the wind generators and reduces the conservativeness of the solution. The historical data is generated using truncated multivariate normal distribution with 100~MW as the mean generation value, and for each partition, the covariance matrix consists of all diagonal entries being 20 and all off-diagonal entries as 16. The random generations are truncated from the below and the above to 80~MW and 120~MW, respectively to avoid infeasibility.
Note that technically wind generators are capable of generating 0~MW to 200~MW which is regarded as their physical bounds. The above data fully specifies the optimization problem~\eqref{eq:SCOPF1}.

We assume that the system operator has access to the samples of the uncertainty in form of historical data. We compare our proposed method with three others that are employed to approximate the solution of the CCSCD problem~\eqref{eq:SCOPF1}. These are explained below:
\begin{enumerate}
    \item \textbf{Worst-Case (WC):} For the worst-case, we solve the CCSCD problem~\eqref{eq:SCOPF1} with the chance-constraint~\eqref{eq:SCOPF1-cc} replaced with robust constraint, where the constraints in~\eqref{eq:SCOPF1-cc} are required to hold for all values of the uncertainty.
    \item \textbf{Scenario:} This is the well known scenario method~\cite{GC-MC:07}, where the constraints in~\eqref{eq:SCOPF1-cc} are imposed for a set of samples of the uncertainty.
    \item \textbf{Robust Optimization with DR Box Uncertainty Set (RO+BoxSet):} This is the approach proposed in~\cite{BP-AH-SB-DC-AC:20} where the set $\Xirob$ used in the robust optimization~\eqref{eq:poly} is restricted to be a hyper-rectangle. This is a special case of our method where the correlation among the wind generators is neglected. 
    \item \textbf{Robust Optimization with DR Polyhedral Uncertainty Set (RO+PolySet):} This is our proposed method, where $\Xirob$ is constructed using Algorithm~\ref{alg:c-xirob}. 
\end{enumerate}
For the above methods,  we compare the cost incurred at the obtained solution and the probability of violating the constraints. The later quantity is obtained by considering a large validation dataset of uncertainty and computing the fraction of samples for which at least one constraint is violated. Note that in the worst-case method, since constraints need to hold for all values of the uncertainty, the violation frequency of the obtained optimal solution is zero. We consider the worst case cost as a benchmark to compare the costs in different methods. All computations are carried out in MATLAB using CVX at the Peregrine HPC cluster of the University of Groningen employing 16GB of memory~\cite{MG-SB:14}.

\begin{figure}[!t]
\centering
\includegraphics[scale=0.35]{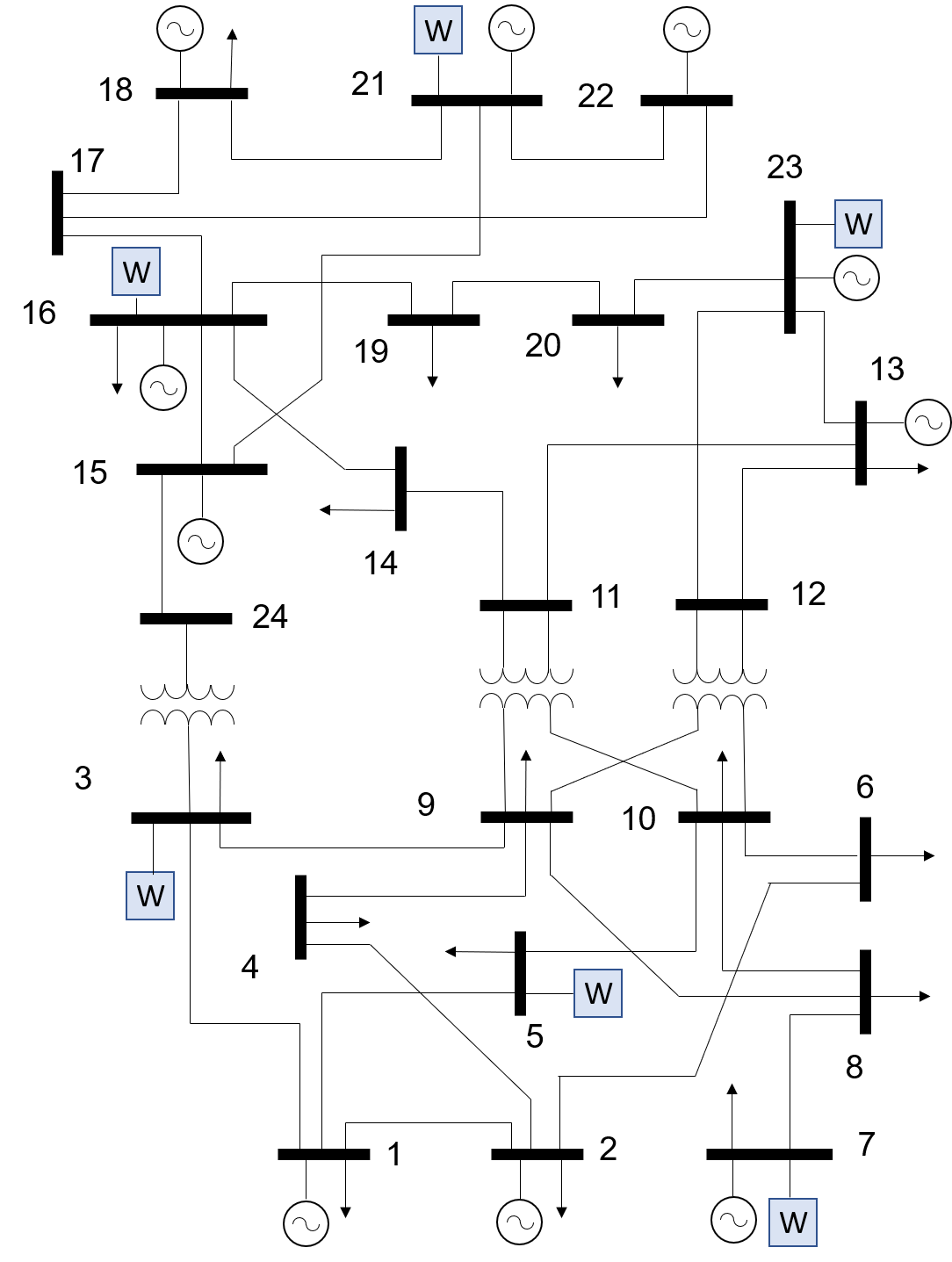}
\caption{\footnotesize The modified IEEE RTS 24-bus system with 12 conventional generators and 6 wind generation units illustrated in light blue. The 17 load locations are represented by arrows.}
\label{fig:24bus}
\end{figure}
In our experiments, we draw 50 samples of the uncertainty from the specified distribution in an i.i.d manner and solve the \textbf{Scenario}, \textbf{RO+BoxSet}, and \textbf{RO+PolySet} methods, where for the later two methods different radii of ambiguity sets are considered for the same set of samples.  We repeat this experiment 50 times. The average values and variances for cost and violation frequency are reported in the Figure~\ref{fig:correlated}. 

Figure \ref{fig:allcorrelated1} shows that for smaller $\theta$ values our method (\textbf{RO+PolySet}) costs less compared to \textbf{RO+BoxSet} while the violation frequency, as shown in Figure \ref{fig:allcorrelated2}, does not change significantly. Therefore, it shows clearly that keeping the safety guarantees same, our method is less conservative than \textbf{RO+BoxSet} for small values of $\theta$. It also shows that both these DR methods result in additional cost and less violation frequency as compared to the \textbf{Scenario} approach, as expected. The advantage of our method is even more clear in Figure~\ref{fig:Pareto} where we plot the trade-off between cost and violation frequency for different value of radii as a Pareto plot. This figure indicates that with the same cost, \textbf{RO+PolySet} guarantees lesser violation frequency as compared to \textbf{RO+BoxSet}. This plot gives the system operator a practical guideline of tuning the radius of the ambiguity set based on the level of importance given to cost-optimality and constraint satisfaction. Increasing $\theta$ results in higher cost but safer dispatch decisions.

We next comment about the computational burden incurred by each method, as outlined in Table~\ref{tab:time}. For smaller number of samples the \textbf{Scenario} approach performs better in terms of computational time compared to the other approaches. However, we recall from earlier discussion that this method performs poorly in terms of constraint violation. For \textbf{RO+BoxSet} and \textbf{RO+PolySet} the variables and constraints mentioned in the table correspond to the robust optimization problem, however, the runtime consists of the time taken to solve the robust as well as all the DR problems. For 50 samples, the runtime for both DR methods is of the same order as the \textbf{Scenario}. However, for 200 samples, the runtime of \textbf{Scenario} is one order higher than the two DR approaches. Comparing \textbf{RO+BoxSet} and \textbf{RO+PolySet}, the former has lower number of variables and constraints and so the runtime is lower than the latter. This computational benefit that \textbf{RO+BoxSet} enjoys comes at the cost of the solution being more conservative when compared to the solution of \textbf{RO+PolySet}, as explained previously. Finally, we comment that the DR problem~\eqref{eqn:DRBounds2} that is solved in a heuristic way and computational times depends on the heuristic algorithm.

\begin{table} 
\scriptsize
\centering
\caption{Data regarding conventional generation units}
\begin{tabular}{lllllllll} 
\hline
Unit \# & Node & $P^{\max}_g$ & $P^{\min}_g$   & \begin{tabular}[c]{@{}l@{}}$R^{\max}_g$\\\end{tabular} & $\costcoeff_g$     & $\costcoeff^+_g(\costcoeff_g^{\out})$ & $H^-_g$  \\ 
\hline
1       & 1    & 152  & 30.4   & 40                                            & 13.32 & 15       & 11  \\
2       & 2    & 152  & 30.4   & 40                                            & 13.32 & 15       & 11  \\
3       & 7    & 350  & 75     & 70                                      & 20.7  & 24       & 16  \\
4       & 13   & 591  & 206.85 & 180                                     & 20.93 & 25       & 17  \\
5       & 15   & 60   & 12     & 60                                            & 26.11 & 28       & 23  \\
6       & 15   & 155  & 54.25  & 30                                            & 10.52 & 16       & 7   \\
7       & 16   & 155  & 54.25  & 30                                            & 10.52 & 16       & 7   \\
8       & 18   & 400  & 100    & 0                                              & 6.02  & 0        & 0   \\
9       & 21   & 400  & 100    & 0                                              & 5.47  & 0        & 0   \\
10      & 22   & 300  & 300    & 0                                              & 0     & 0        & 0   \\
11      & 23   & 310  & 108.5  & 60                                            & 10.52 & 14       & 8   \\
12      & 23   & 350  & 140    & 40                                            & 10.89 & 16       & 8   \\
\hline
\end{tabular}
\label{tab:gen-par}
\end{table}

\begin{figure*} [!t]
\centering
\begin{subfigure}{.4\textwidth}
  \centering
  \includegraphics[scale=0.5]{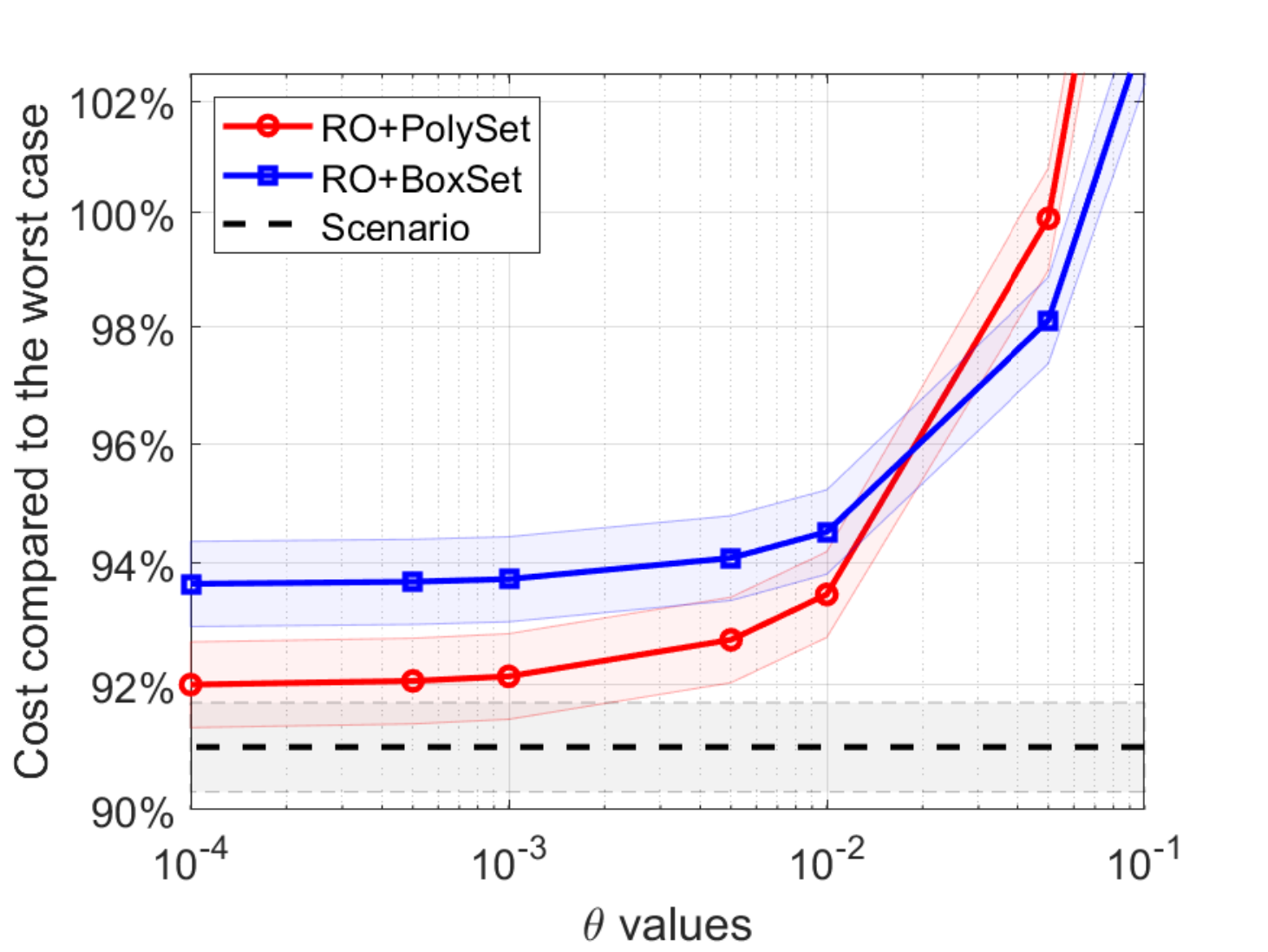}
  \caption{Costs}
  \label{fig:allcorrelated1}
\end{subfigure}
\begin{subfigure}{.4\textwidth}
  \centering
  \includegraphics[scale=0.5]{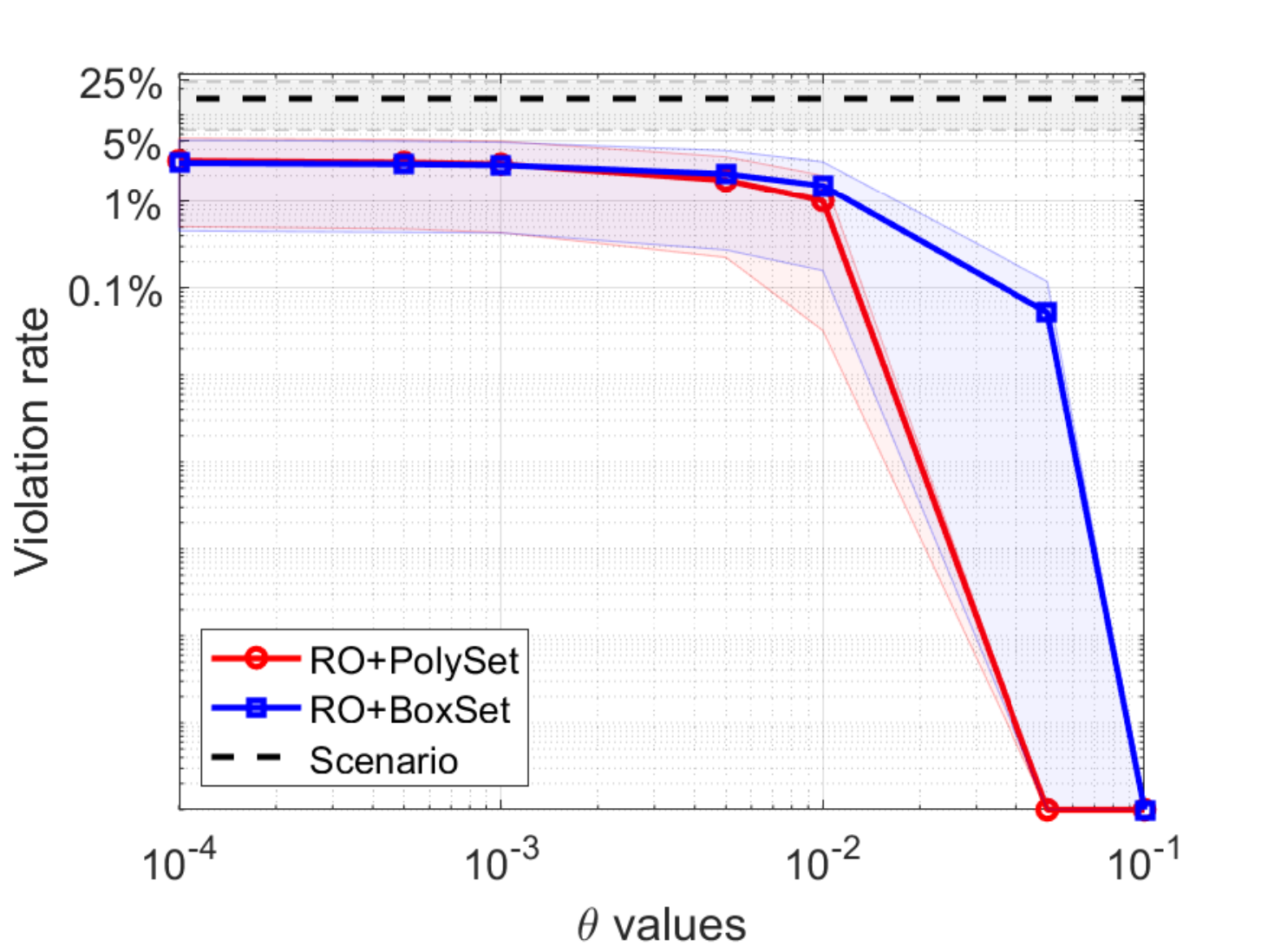}
  \caption{Violation}
  \label{fig:allcorrelated2}
\end{subfigure} \caption{Operational cost and violation frequency of constraints for different ambiguity sets varying in radius~$\theta$. This experiment is repeated 50 times. In each iteration, the problem is solved with a set of 50 correlated samples, using all three methods with $\theta = \{0.0001, 0.0005, 0.001, 0.005, 0.01, 0.05, 0.1\}$.  The shaded regions represent the standard deviation of the results around the average results.}
\label{fig:correlated}
\end{figure*}
\begin{figure}[!t]
\centering
\includegraphics[scale=0.5]{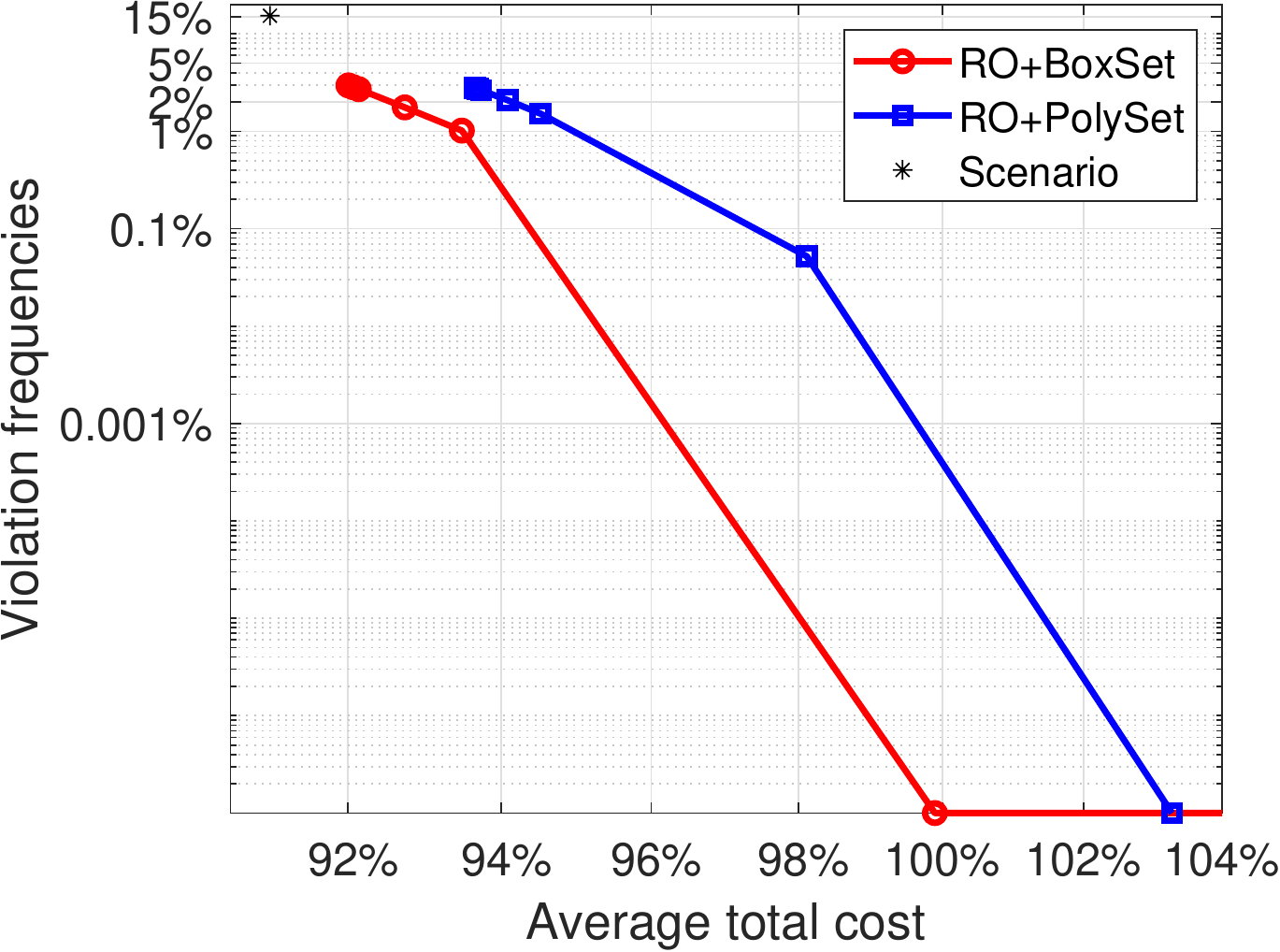} \caption{Pareto front for IEEE RTS 24-bus system obtained by plotting the cost and the violation frequency at different radii. Larger value of the radius results in higher cost and lower violation. The Scenario method does not have such a tuning parameter and hence is depicted as a point.}
\label{fig:Pareto}
\end{figure}
\begin{table}
\centering
\caption{Computation time and memory usage}
\begin{tabular}{lllll} 
\hline
Method   & Samples              & Variables & Constraints & Runtime  \\ 
\hline
Scenario & 50  & 5088      & 220227      & 355.06 s   \\
RO+BoxSet      & 50                      & 38280     & 62119       & 461.86 s   \\
RO+PolySet      & 50                      & 63304     & 87143       & 743.06 s  \\ 
\hline
Scenario & 200 & 5088      & 876627     & 5629.93 s \\
RO+BoxSet      & 200                     & 38280     & 62199       & 453.25  s \\
RO+PolySet      & 200                     & 63304     & 87143       & 775.81 s  \\
\hline
\end{tabular}
\label{tab:time}
\end{table}
\section{Conclusion}
We studied chance-constrained security-constrained dispatch and developed an algorithm to solve its DR counterpart where Wasserstein ambiguity sets were used. An attractive feature of our algorithm is the computational ease of approximating a solution, even for large networks. This was facilitated by adopting a two-step approach of solving the DR problem. We demonstrated the cost-robustness trade-offs of our results on stylized IEEE test case. In future, we wish to explore distributed algorithms to solve the DR problems formulated here. We also aim to reduce the conservativeness of our approach by adopting a one-step procedure that combines uncertainty quantification and robust optimization. 

\section*{Acknowledgments}
The last author would like to thank Dr. B.K. Poolla, Dr. A.R. Hota, and Dr. S. Bolognani for useful discussions regarding the optimization problem and the algorithmic framework.

\bibliographystyle{IEEEtran}
\bibliography{Reference}

\end{document}